\documentclass{article}
\usepackage{graphicx, amscd, amsmath,amssymb, amstext,amsthm,amsfonts, enumitem, comment, breqn}
\usepackage{mathrsfs}
\usepackage[utf8]{inputenc}
\usepackage[T1]{fontenc}
\usepackage{hyperref}
\usepackage[a4paper, left=2.6cm, right=2.6cm, top=3cm, bottom=2.5cm]{geometry}
\usepackage{authblk}
\usepackage{marvosym}
\usepackage{xcolor}
\usepackage{caption}

\newcommand{\eps}{\varepsilon}

\newcommand{\SU}{{\mathcal U}}

\newcommand{\N}{\mathbb{N}}

\newcommand{\R}{\mathbb{R}}

\newcommand{\seqx}{\underline{x}}

\newcommand{\seqz}{\underline{z}}
\newcommand{\orbx}{\seqx_T}

\newcommand{\dist}{\rho}

\newcommand{\MT}{\mathcal{M}_T}

\newtheorem{thm}{Theorem}[section] 
\newtheorem{cor}[thm]{Corollary}
\newtheorem{prop}[thm]{Proposition}

\newtheorem{question}[thm]{Question}

\theoremstyle{definition}
\newtheorem{defn}[thm]{Definition}

\newtheorem{example}[thm]{Example}

\title{The interplay between partial specification, average shadowing, and Besicovitch completeness}
\author{Melih Emin Can, Marcin Kulczycki}
\date{\today}

\begin{document}
\maketitle

\begin{abstract}
Let $(X,T)$ be a compact dynamical system. This article proves that if $(X,T)$ has the partial specification property, then it has the average shadowing property. It is also proven that if $(X,T)$ is surjective and has the partial specification property, then the set of ergodic measures of $(X,T)$ is dense in the space of its invariant measures. An example of a compact dynamical system that is not Besicovitch complete is also given.
\end{abstract}

\section{Introduction}

The specification property was introduced by Rufus Bowen in 1971 in \cite{Bowen}. The main idea of this new property was, roughly speaking, for finite collections of finite orbits to be traced by a true orbit of the system as long as the separation between finite orbits was large enough. It was initially used in the context of Axiom A diffeomorphisms.

Systems possessing this property are, by necessity, complex and chaotic in nature - it is known that, for example, they have to be topologically mixing and they have positive topological entropy. The flagship examples of systems with specification are mixing interval maps and mixing shifts of finite type.

Over time, it became clear that the very useful notion of specification is limited by the extent of the family of maps possessing it. Several authors have successfully defined its subtler, weaker variants that would apply to a wider range of systems. We refer the reader to \cite{KLO} for a comprehensive overview, while reproducing below a graph from that paper showcasing various variants of specification. It is now known that directed paths in this graph represent the only implications between these properties that hold.

\begin{figure}
    \centering
    \includegraphics[width=0.9\linewidth]{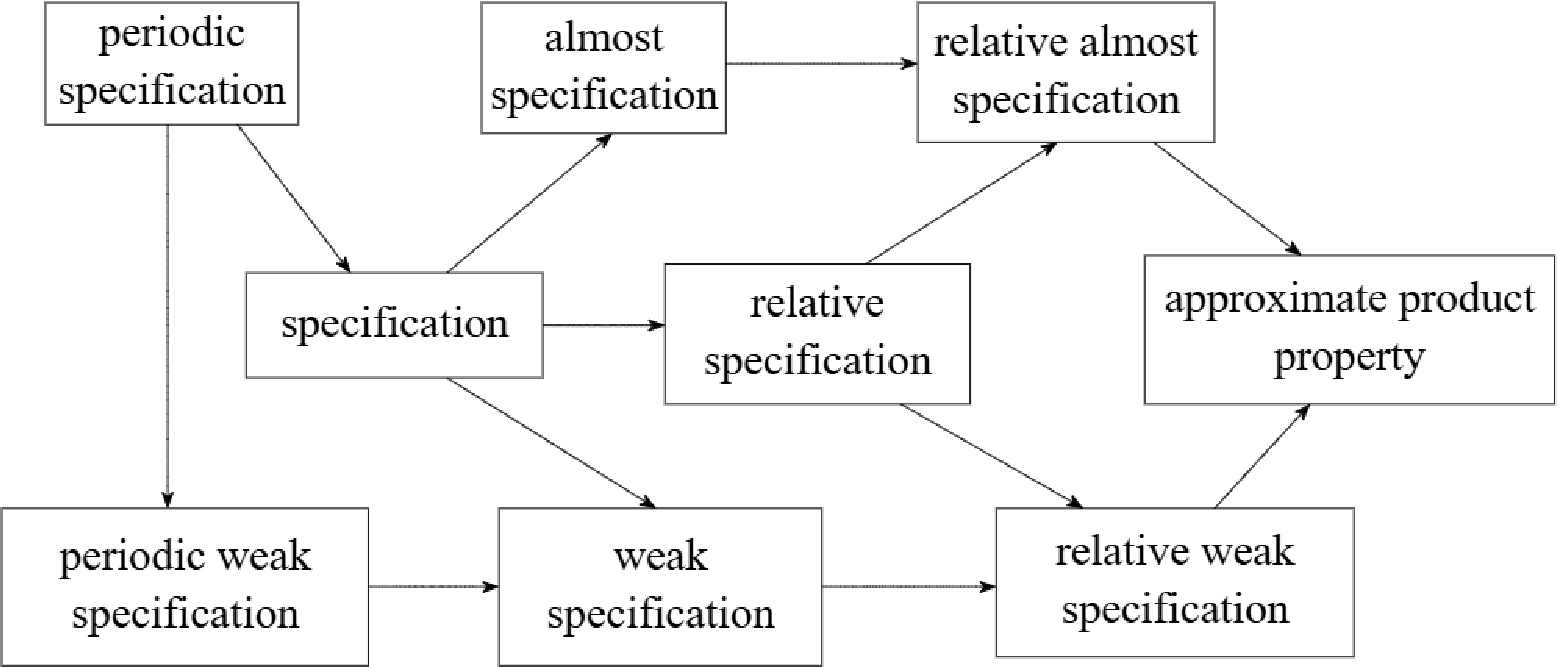}
    \caption*{Figure 1 (reproduced from \cite{KLO})}
\end{figure}

Another addition to this list, the partial specification, was made in 2025 by Yaari \cite{Yaari}, who worked in the context of metrizable abelian groups. The novelty of this concept is that the tracing of segments of orbits is allowed to have a certain percentage of errors. In this way, partial specification is related to specification in a somewhat similar way to how average shadowing is related to classic shadowing. The definition itself has additional strict assumptions about the periods of tracing points, which is a result of it being custom-tailored for work in the setting of groups, which is of central interest to Yaari.

We feel that partial specification is of interest by itself, as it fits in a natural way with other already existing properties. It might actually be prudent to consider if it should be integrated into the context of classical discrete dynamical systems without the assumptions on the periodicity of tracing points, in particular in light of the fact that we did not need them to obtain Proposition \ref{prop:infinite-partial-spec}, which goes along the lines of what \cite[Lemma 17]{KLO2} does for weak specification. For now, however, we have kept partial specification, in Definition \ref{def:partial-specification}, exactly as proposed by Yaari.

In parallel with specification and its variants, another family of properties of dynamical systems came into existence: shadowing and its variants. Introduced in its classic form by Blank in \cite{Blank}, it has since seen a prolific growth. The reader may find \cite{Pilyugin} and \cite{Pilyugin2} useful general references, while \cite{Blank2} illustrates how a modern take on the subject may look. These properties come, in general, in two flavors: strict epsilon-delta variants (e.g. classic shadowing) that demand exact tracing along the whole time interval that is of interest, and average variants (e.g. average shadowing) that are interested in the average error of tracing, which allows for large misalignments, as long as they are sparse enough. Of particular interest is the idea of asymptotic average shadowing, a strong property introduced by Gu in \cite{Gu}, which is nevertheless satisfied by a wide class of maps, including all mixing maps of the unit interval.

Both specification and shadowing tend to cohabit the same classes of maps, and the interplay between them is now somewhat understood, while there still remains a rich playing field full of open problems. For example, \cite[Theorem 3.5]{KKO} proves that a surjective map with almost specification has the asymptotic average shadowing property, and \cite[Theorem 3.8]{KKO} proves, under the assumption of classic shadowing, the equivalence of topological mixing, specification, almost specification, average shadowing, and asymptotic average shadowing. Our paper starts to place the almost specification in the big picture by proving, in Theorem \ref{thm:partial-spec=>asp}, that partial specification implies the average shadowing property. In addition, Proposition \ref{prop:infinite-partial-spec} might be of future interest as a useful technical tool, much like \cite[Lemma 17]{KLO2} works for weak specification.

As an application of Theorem \ref{thm:partial-spec=>asp} we obtain, in Corollary \ref{cor:partial-spec=>measure-density}, the density of ergodic measures in the space of invariant measures for surjective systems with the partial specification property. This result relies on digging into technical details of \cite{BJK} to show that, although not stated explicitly, it contains the proof of Theorem \ref{thm:asp=>measure-density}.

The final subject we touch on is the notion of Besicovitch completeness, introduced in \cite{BFK}. It is known to have close relations with shadowing (e.g., \cite[Theorem 4.4]{CT} proves that, for surjective maps, the asymptotic average shadowing property is equivalent to the average shadowing property plus Besicovitch completeness). It has been known from the start that the full shift is Besicovitch complete \cite[Proposition 2]{BFK}, and this property is, in fact, so ubiquitous that to the best of our knowledge, no example of a system not having it is known to this day. We present an example of a compact dynamical system that is not Besicovitch complete in Example \ref{exm:not-Besicovitch-complete}.




\section{Basics}
\subsection{The Besicovitch Pseudo-Metric}  
First, we define the notions of upper and lower asymptotic densities for the subsets of natural numbers.
\begin{defn}
    Let $A$ be a subset of $\N$. We define the upper asymptotic density of $A$ by 
    \begin{equation}
        d^*(A) = \limsup_{n\rightarrow \infty}\dfrac{\left|A\cap \{1,2,\ldots,n\}\right|}{n}.
    \end{equation}
    Similarly, we denote the lower asymptotic density of $A$ by 
    \begin{equation}
        d_*(A) = \liminf_{n\rightarrow \infty}\dfrac{\left|A\cap \{1,2,\ldots,n\}\right|}{n}.
    \end{equation} 
    If we have $d^*(A)=d_*(A)$ then this limit value is called the density of $A$ and is denoted by $d(A)$.
\end{defn}

In what follows, it is assumed that $(X,\rho )$ is a compact metric space with diameter 1. Let us denote the product space by 
\[X^\infty=\left\{ \underline{x}=\{x_j\}_{j=0}^\infty \mid x_j \in X \text{ for all } j\in \N\right\}.\] An element $\underline{x}\in X^\infty$ is called an $X$-valued sequence. Obviously, $X^\infty$ is a compact metric space. We denote the shift operator on $X^\infty$ by $S$ and recall that it is given by $S(\{x_j\}_{j=0}^\infty)=\{x_{j+1}\}_{j=0}^\infty$.

Let $T\colon X\rightarrow X$ be a continuous map.
Recall that the orbit of a point $x\in X$ is the sequence $\theta(x)=\{ T^n(x) \mid n\in \N \}$. We define
\begin{equation}\label{defn:space-of-orbits}
    X_T = \{\seqx_T \in X^\infty : x \in X\}.
\end{equation}
\begin{defn}\label{defn:Besicovitch-pseudometric}
Let $\underline{x}=\{x_j\}_{j=0}^\infty\in X^\infty$ and $\underline{z}=\{z_j\}_{j=0}^\infty \in X^\infty$.
    The Besicovitch pseudo-metric on $X^\infty$ is defined by 
    \begin{equation}\label{eqn:defn-Besicovitch}
        \rho_B\left( \underline{x}, \underline{z} \right)= \limsup\limits_{n\to\infty} \dfrac{1}{n}\sum_{j=0}^{n-1} \rho(x_j, z_j).
    \end{equation}
\end{defn}
Let $\seqx,\seqz \in X^\infty$. We define the metric $\pi$ on $X^\infty$ by
\[
    \pi( \seqx,\seqz )=\sup\left\{ \min \left(\rho(x_j,z_j),\dfrac{1}{j+1} \right) \mid j\in \N\right\}. 
\]
Observe that the topology on $X^\infty$ which is induced by the metric $\pi$ is compatible with the product topology on $X^\infty$. Now we can consider a dynamical pseudo-metric on $(X^\infty,S)$ as follows.

\begin{defn}\label{def:dynamical-Bp-on-product}
    For $\seqx,\seqz \in X^\infty$ we define the dynamical Besicovitch pseudo-metric on $X^\infty$ by 
    \begin{equation}\label{eqn:def-of-pi_B}
        \pi_B(\seqx,\seqz)=\limsup\limits_{n\to\infty}\dfrac{1}{n}\sum_{j=0}^{n-1}\pi\left(S^j(\seqx),S^j(\seqz)\right).
    \end{equation}
\end{defn}

By \cite[Lemma 5.3]{BJK} we observe that pseudo-metrics $\rho_B$ and $\pi_B$ are uniformly equivalent.

The Besicovitch pseudo-metric on $X^\infty$ induces a pseudo-metric on a topological dynamical system $(X,T)$ in the following way: let $x,y\in X$ and define
\begin{equation}\label{eqn:defn-dynamic-Besicovitch}
    \rho_B^T(x,y)=\rho_B(\{T^j(x)\}_{j=0}^\infty,\{T^j(y)\}_{j=0}^\infty)=\limsup\limits_{n\to\infty}\dfrac{1}{n}\sum_{j=0}^{n-1}\rho(T^j(x),T^j(y)).
\end{equation}
We call $\rho_B^T$ the (dynamical) Besicovitch pseudo-metric on $(X,T)$. Using \eqref{eqn:defn-Besicovitch} and \eqref{eqn:defn-dynamic-Besicovitch} we see that
$\rho_B^T(x,y)=\rho_B(\orbx,\underline{y}_T)$ for $x,y\in X$.




\begin{defn} 
        Let $(X,T)$ be a topological dynamical system.
		A sequence $\{z^{(n)}\}_{n=0}^\infty \subset X$ is called Besicovitch-Cauchy if for every $\eps>0$ there is $N \in \N$ such that for every $n,m \geq N$ we have
		$$\rho_B^T(z^{(n)},z^{(m)}) <\eps.$$
	\end{defn}

	\begin{defn}\label{Besicovitchshadowing}
		A topological dynamical system $(X,T)$ is called  Besicovitch complete if for every Besicovitch-Cauchy sequence $\{z^{(n)}\}_{n=0}^\infty \subset X$ there is a point $z \in X$ such that
		$$\lim_{n \to \infty }\rho_B^T(z^{(n)},z) = 0.$$
	\end{defn}

\subsection{Pseudo-orbits, tracing properties, and specification}

\begin{defn}\label{def:delta-p.o}
    For $\delta>0$, a sequence $\{x_i\}_{i=0}^{\infty}\in X^\infty$ is called a $\delta$-pseudo-orbit if $\rho(T(x_i),x_{i+1})<\delta$ for all $i\geq 0$. 

    For $n\in\N$ a finite sequence $\{x_i\}_{i=0}^{n-1}\in X^n$ satisfying $\rho(T(x_i),x_{i+1})<\delta$ for all $0\leq i\leq n-2$ is called a $\delta$-chain.
\end{defn}
We denote the space of $\delta$-pseudo-orbits in $X^\infty$ by $X_T^{\delta}$, i.e, 
\[
    X_T^{\delta}=\left\{ \seqx=\{x_i\}_{i=0}^\infty \in X^\infty \mid \rho(T(x_i),x_{i+1})<\delta \text{ for all } i\geq0 \right\}\subset X^\infty.
\]

Blank \cite{Blank} analyzed a type of pseudo-orbits referred to as average pseudo-orbits and used them to define the average shadowing property:

\begin{defn}\label{delta-APO}
		We call a sequence $\seqx \in X^\infty$ a $\delta$-average pseudo-orbit for $T$ if there is $N \in \N$ such that for every $n \geq N$ and $k \geq 0$ we have $$ \frac{1}{n}\sum_{i=0}^{n-1}\dist(T(x_{i+k}), x_{i+k+1}) < \delta.$$ 
\end{defn}

\begin{defn}\label{def_asp}
		A topological dynamical system $(X, T)$ has the average shadowing property if for every $\eps >0$ there is $\delta>0$ such that for every $\delta$-average pseudo-orbit $\seqx \in X^\infty$ for $T$ there is $z \in X$ satisfying
		$$\rho_B(\seqz_T, \seqx) < \eps.$$
\end{defn} 

 Later, Gu \cite{Gu} investigated the case where the average error over the sequence of indices from $0$ to $n-1$ tends to zero as $n$ tends to infinity and used it to define the asymptotic average shadowing property:

\begin{defn}\label{AAPO}
		A sequence $\seqx \in X^\infty$ is called an asymptotic average pseudo-orbit for $T$ if $$\limsup_{n \to \infty} \frac{1}{n}\sum_{i=0}^{n-1}\dist(T(x_i), x_{i+1}) = 0.$$ 
\end{defn}

	\begin{defn}\label{def_aasp}
		A topological dynamical system $(X, T)$ has the asymptotic average shadowing property if for every asymptotic average pseudo-orbit $\seqx \in X^\infty$ for $T$ there is $z \in X$ with
		$$\rho_B(\seqz_T, \seqx)= 0.$$
	\end{defn}

Kamae in \cite{Kamae} studied a previously unnamed family of sequences in $X^\infty$, which has been called vague pseudo-orbits in \cite[Definition 3.9]{CT}, and used it to define the vague specification property. Kamae's work \cite{Kamae} was published much earlier than the works of Gu and Blank, but, unfortunately, it went mostly unnoticed:
 
\begin{defn}\label{VPO}
		A sequence $\seqx \in X^\infty$ is said to be a vague pseudo-orbit for $T$ if for every open neighborhood $\SU$ of $X_T$ in $X^\infty$ we have
		$$d(\{ n \in \N : S^n(\seqx)\in \SU\})=1.$$
\end{defn}
    
\begin{defn}\label{def_vsp}
		A topological dynamical system $(X, T)$ has the vague specification property if for every vague pseudo-orbit $\seqx \in X^\infty$ for $T$ there exists $z \in X$ satisfying that for every $\eps>0$ we have
		$$d\left(\{ n \in \N_0 : \rho(T^n(z), x_n)<\eps \}\right)=1.$$
\end{defn}
\begin{thm}\cite[Theorem 4.6]{CT} 
     A topological dynamical system $(X,T)$ has the vague specification property if and only if $(X,T)$ has the asymptotic average shadowing property. 
\end{thm}
\begin{thm}\cite[Theorem 4.4]{CT}
    Let $(X,T)$ be a surjective topological dynamical system. Then $(X,T)$ has the asymptotic average shadowing property if and only if $(X,T)$ has the average shadowing property and $(X,T)$ is Besicovitch-complete.
\end{thm}

The notions of partial shadowing and partial specification (Definition \ref{def:partial-specification}) were introduced by Yaari in \cite{Yaari}:

\begin{defn}\label{def:delta-partial-p.o}
    Let $\delta>0$. A sequence $\{x_i\}_{i=0}^r\subset X$ for $r \geq 1$ is called a $\delta$-partial pseudo-orbit for $T$ if 
    \[
        \dfrac{1}{r}\left | \left\{0\leq i \leq r-1 \mid \rho(T(x_i),(x_{i+1}))<\delta\right\} \right |>1-\delta.
    \]
\end{defn}

\begin{defn}\label{def:partial-tracing-of-partial-p.o}
    Let $\eps>0$. A sequence $\{x_i\}_{i=0}^r$ is $\eps$-partially traced by a point $z\in X$ if 
    \[
        \dfrac{\left |\Lambda \right |}{r+1}>1-\eps \quad \text{where} \quad \Lambda=\{ 0\leq i \leq r \mid \rho(T^i(z),x_i)<\eps\}.
    \]
\end{defn}
\begin{defn}\label{def:partial-shadowing}
    A topological dynamical system $(X,T)$ has the partial shadowing property if for every $\eps>0$ there is $\delta>0$ such that every $\delta$-partial pseudo-orbit can be $\eps$-partially traced by some point in $X$.
\end{defn}

Another class of actively studied properties of dynamical systems are various variants of specification.

\begin{defn}
    Given $a,b\in \N$ with $a\leq b$ we define two types of orbit segment of $x\in X$:
    \begin{equation*}\begin{array}{l}
        T^{[a,b]}(x)=T^a(x),T^{a+1}(x),\ldots, T^{b}(x),\\
        T^{[a,b)}(x)=T^a(x),T^{a+1}(x),\ldots, T^{b-1}(x).
    \end{array}\end{equation*}
\end{defn}

\begin{defn}\label{def:M-spaced-specification}
    Let $f\colon \N \to \N$ be a function, let $r \geq 1$, and let $x_1,\ldots,x_r\in X$.
    
    A collection of orbit segments $\{T^{[a_i,b_i]}(x_i)\}_{i=1}^r$ is called an $f$-spaced specification if $a_1<b_1<a_2<b_2<\ldots<a_r<b_r$ and $a_{i}-b_{i-1}\geq f(b_i-a_i+1)$ for $i=2,\ldots,r$.

    Similarly, a collection of orbit segments $\{T^{[a_i,b_i)}(x_i)\}_{i=1}^r$ is called an $f$-spaced specification if $a_1<b_1<a_2<b_2<\ldots<a_r<b_r$ and $a_{i}-b_{i-1}\geq f(b_i-a_i)$ for $i=2,\ldots,r$.
    
      In particular, if $M\in \N$ then by an $M$-spaced specification (of either kind) we understand the $f$-spaced specification, where $f$ is constantly equal to $M$.
\end{defn}

\begin{defn}
    A specification $\{T^{[a_i,b_i]}(x_i)\}_{i=1}^{r}$ is $\eps$-traced by a point $y \in X$ if 
    \[
        \rho(T^j(x_i),T^j(y))<\eps \quad \text{for $1\leq i \leq r$ and $a_i\leq j \leq b_i$}.
    \]
\end{defn}
\begin{defn}\label{def:specification-property}
    A topological dynamical system $(X,T)$ has the specification property if for every $\eps>0$ there exists $M\in \N$ such that every $M$-spaced specification $\{T^{[a_i,b_i]}(x_i)\}_{i=1}^r$ can be $\eps$-traced by some point $y\in X$.
\end{defn}
\begin{defn}\label{def:weak-specification-property}
    A topological dynamical system $(X,T)$ has the weak specification property if for every $\eps>0$ there exists a non-decreasing function $f:\N \to \N$ with $\lim_{n\to\infty}f(n)/n \to 0$ such that every $f$-spaced specification $\{T^{[a_i,b_i]}(x_i)\}_{i=1}^r$ can be $\eps$-traced by some point $y\in X$.
\end{defn}
It follows from the definitions that the specification property implies the weak specification property. 
\begin{defn}\label{def:partial-tracing-of-specification}
    Let $\eps>0$ and let $M\geq 1$. An $M$-spaced specification $\{T^{[a_i,b_i)}(x_i)\}_{i=1}^r$ is $\eps$-partially traced by a point $z\in X$ if for every $1\leq i\leq r$ we have
    \[ 
        \dfrac{\left| \Lambda_i \right|}{b_i-a_i}>1-\eps \quad \text{where} \quad \Lambda_i=\{ a_i\leq n <b_i \mid \rho(T^n(z),T^n(x_i))<\eps \}.
    \]    
\end{defn}
\begin{defn}\label{def:partial-specification-of-periods-P}
    Let $P\subset \N$. A topological dynamical system $(X,T)$ has the partial specification property with periods $P$ if for every $\eps>0$ there exist $M,N\in \N$ such that for every $M$-spaced specification $\{T^{[a_i,b_i)}(x_i)\}_{i=1}^r$ and every $n\in P$  with $n\geq \max\{N,(1+\eps)b_r\}$ there exists a point $z\in X$ of period $n$ which $\eps$-partially traces the specification $\{T^{[a_i,b_i)}(x_i)\}_{i=1}^r$.
\end{defn}
\begin{defn}\label{def:partial-specification}
    A topological dynamical system $(X,T)$ has the partial specification property if for every $c\in \N$ there is $P\subset c\N$ such that $(X,T)$ has the partial specification property with periods $P$.
\end{defn}

\section{Main Results}

Let us first recall some widely known dynamical concepts.
\begin{defn}
    We say that a topological dynamical system $(X,T)$ is chain transitive if for every $\delta>0$ and for every $x,y\in X$ there are $n\geq 1$ and a $\delta$-chain $(x_i)_{i=0}^{n-1}$ with $x_0=x$ and $x_{n-1}=y$.

    We say that $(X,T)$ is chain mixing if for every $\delta>0$ there exists $M\in \N$ such that for every $x,y\in X$ and for all $n\geq M$ there is a $\delta$-chain $(x_i)_{i=0}^{n-1}$ with $x_0=x$ and $x_{n-1}=y$.

    We say that $(X,T)$ is topologically mixing if for every pair of nonempty open sets $U,V\subset X$ there exists $N\in\N$ such that for all $n\geq N$ we have $T^n(U)\cap V\neq\emptyset$.
\end{defn}
\begin{prop}\label{prop:partial-p.o-avg-p.o}
    Let $(X,T)$ be a topological dynamical system and let $\delta>0$. Then for every $\delta^2$-average pseudo-orbit $(x_i)_{i=0}^\infty$ of $(X,T)$ there exists $N\in \N$ such that for all $n\geq N$ and $k\geq 0$ the sequence $(x_{i+k})_{i=0}^{n}$ is a $\delta$-partial pseudo-orbit of $(X,T)$. 
\end{prop}

\begin{proof}
    Let $(x_i)_{i=0}^\infty\subset X$ be a $\delta^2$-average pseudo-orbit. Then, there is $N\in \N$ such that for all $n\geq N$, $k\geq0$ we have
    \[
        \dfrac{1}{n}\sum_{i=0}^{n-1}\rho\left( T(x_{i+k}),x_{i+k+1} \right)<\delta^2.
    \]
    Fix $k\geq 0$, and $n\geq N$. Let us write 
    \[
        J_\delta(n,k) = \left\{ 0\leq i < n \mid \rho\left(T(x_{i+k}),x_{i+k+1}\right)<\delta \right\}.
    \] 
    It is enough to show that $\left|J_\delta(n,k)\right|>(1-\delta)n$. Suppose to the contrary that $\left|J_\delta(n,k)\right|\leq (1-\delta)n=n-\delta n$. That means, for at least $\lceil n\delta\rceil$ elements $i$ of $\{0,1,\ldots,n-1\}$ we have $\rho(T(x_{i+k}),x_{i+k+1})\geq \delta$. It follows that
    \[
        \dfrac{1}{n}\sum_{i=0}^{n-1}\rho\left( T(x_{i+k}),x_{i+k+1} \right)\geq \delta^2,
    \] which is a contradiction. Hence, $\left| J_\delta(n,k)\right|>(1-\delta)n$.
\end{proof}

\begin{prop}\label{prop:prod-space-partial-shadow}
    If $(X,T)$ has the partial shadowing property, then $(X\times X, T\times T)$ has the partial shadowing property.
\end{prop}
\begin{proof}
    First, we note that $\rho_\max((x_1,y_1),(x_2,y_2))=\max\left\{\rho(x_1,x_2),\rho(y_1,y_2)\right\}$ is a metric on $X\times X$ compatible with the product topology. Fix $\eps>0$ and let $\delta>0$ be such that every $\delta$-partial pseudo-orbit can be $\eps/3$-traced in $X$. Let $(x_n,y_n)_{n=1}^N\in X\times X$ be a $\delta$-partial pseudo-orbit of $T\times T$, which means that
    \begin{equation}\label{eqn:partial-p.o-in-prod-space}
        \dfrac{1}{N-1}\left|\left\{ 1\leq n \leq N-1 \mid \rho_\max\left( (T\times T)((x_n,y_n)),(x_{n+1},y_{n+1})\right)<\delta \right\} \right|>1-\delta.
    \end{equation}
    Using \eqref{eqn:partial-p.o-in-prod-space} and the definition of $\rho_\max$ we get 
    \begin{align*}
         &\dfrac{1}{N-1}\left|\left\{ 1\leq n \leq N-1 \mid \rho\left( T(x_n),x_{n+1}\right)<\delta \right\} \right|>1-\delta,\\
         &\dfrac{1}{N-1}\left|\left\{ 1\leq n \leq N-1 \mid \rho\left( T(y_n),y_{n+1}\right)<\delta \right\} \right|>1-\delta.
    \end{align*}
 Therefore, $(x_n)_{n=1}^N$ and $(y_n)_{n=1}^N$ are $\delta$-partial pseudo-orbits of $T$.
    Hence, there exists $x',y'\in X$ such that 
     \begin{align*}
         &\dfrac{1}{N}\left|\left\{ 1\leq n \leq N \mid \rho\left( T^{n-1}(x'),x_{n}\right)<\eps/3 \right\} \right|>1-\eps/3,\\
         &\dfrac{1}{N}\left|\left\{ 1\leq n \leq N \mid \rho\left( T^{n-1}(y'),y_{n}\right)<\eps/3 \right\} \right|>1-\eps/3.
    \end{align*}
    Thus,
    \[
        \dfrac{1}{N}\left|\left\{ 1\leq n \leq N \mid \rho_\max\left( (T\times T)^{n-1}((x',y')),(x_n,y_{n})\right)<\eps \right\} \right|>1-\eps.\qedhere
    \]
\end{proof}
\begin{prop}\label{prop:partial-shadow=>chain-mixing}
    If $(X,T)$ is a surjective topological dynamical system with the partial shadowing property, then $(X,T)$ is chain mixing.
\end{prop}
\begin{proof}
    It is known that if $T\times T$ is chain transitive, then $T$ is chain mixing (see the remark after Corollary 12 in \cite{RW}). By this fact and by Proposition~\ref{prop:prod-space-partial-shadow} it is enough to show that $(X,T)$ is chain transitive.
    
      Fix $1>\delta>0$ and $x,y\in X$. Using uniform continuity of $T$ we find $\gamma>0$ such that if $z,z'\in X$ and $\rho(z,z')<\gamma$, then $\rho(T(z),T(z'))<\delta$. Without loss of generality, we may assume that $\delta>\gamma$ and $\dfrac{1}{4}>\gamma$. Next, we use partial shadowing to find $\gamma>\beta>0$ such that every $\beta$-partial pseudo-orbit is $\gamma$-partially traced by some point in $X$.
    
      Take $l>0$ such that 
     $\beta>\dfrac{1}{2l-1}$ and, using the surjectivity of $T$, find $y_0$ in $X$ such that $T^{l-1}(y_0)=y$. Observe that $(x_i)_{i=0}^{2l-1}$ defined as
    \[
        x_0=x, x_1=T(x),\ldots,x_{l-1}=T^{l-1}(x),x_l=y_0, x_{l+1}=T(y_0),\ldots,x_{2l-1}=y
    \]
    is a $\beta$-partial pseudo-orbit. Hence, there is $z\in X$ such that 
    \begin{equation}\label{eqn:tracing-gammafraction-of-orbit}
        \dfrac{1}{2l}\left| \left\{ 0\leq i \leq 2l-1 \mid \rho(x_i,T^{i}(z))<\gamma \right\}\right|>1-\gamma.
    \end{equation}
  
    Using $\gamma<\dfrac{1}{4}$ and \eqref{eqn:tracing-gammafraction-of-orbit} we pick $i$ such that $0\leq i \leq l-1$ and $\rho(T^i(z),T^i(x))<\gamma$. Similarly, we pick $j$ such that $l\leq j \leq 2l-1$ and $\rho(T^j(z),T^{j-l}(y_0))<\gamma$. We will now investigate four possible cases, listing for each a $\delta$-pseudo-orbit connecting $x$ to $y$: 
    \begin{itemize}
        \item Case $1$: we have $1\leq i \leq l-1$ and $l\leq j \leq l-2$. Consider
    \[
        x,T(x),\ldots,T^{i-1}(x),T^i(z),\ldots,T^{j-1}(z),T^{j-l}(y_0),\ldots ,T^{2l-2}(y_0),y.
    \]

        \item Case $2$: we have $i=0$ and $l\leq j \leq 2l-1$. Consider
    \[
        x,T(z),\ldots,T^{j-1}(z),T^{j-l}(y_0),\ldots,T^{2l-2}(y_0),y.
    \]
        
        \item Case $3$: we have $1\leq i \leq l-1$ and $j=2l-1$. Consider
    \[
        x,T(x),\ldots,T^{i-1}(x),T^i(z),\ldots,T^{2l-2}(z),y.
    \] 
        
        \item Case $4$: we have $i=0$ and $j=2l-1$. Consider
    \[
     x,T(z),T^2(z),\ldots,T^{2l-2}(z),y.
    \]
        
    \end{itemize}

\end{proof}
The next result is a direct consequence of Proposition \ref{prop:partial-shadow=>chain-mixing} and \cite[Lemma 4.7]{Yaari}. 
\begin{cor}
    If $(X,T)$ is a surjective topological dynamical system with the partial specification property, then $(X,T)$ is chain mixing.
\end{cor}
It is known that chain mixing is equivalent to topological mixing for systems with the shadowing property \cite{Walters1978}. This, combined with Proposition \ref{prop:partial-shadow=>chain-mixing} and \cite[Theorem 45]{KLO}, provides us with the following:
\begin{cor}\label{cor:Partial-shadowing=>Specification}
    Let $(X,T)$ be a surjective topological dynamical system with the shadowing property and the partial shadowing property. Then $(X,T)$ has the specification property.
\end{cor}
\begin{prop}\label{prop:partial-spec=>top.mixing}
    Let $(X,T)$ be a surjective topological dynamical system with the partial specification property. Then $(X,T)$ is topologically mixing.
\end{prop}

\begin{proof}
    Let $U,V\subset X$ be nonempty open sets. Fix $x\in U$ and $y\in V$. Let $1/2>\eps>0$ be such that 
    \[
    B(x,\eps)=\{z\in X \mid \rho(x,z)<\eps\} \subset U \quad \text{and} \quad B(y,\eps)=\{z\in X \mid \rho(z,y)<\eps\}\subset V.
    \]
    Let $M\geq 1$ be such that any $M$-spaced specification can be $\eps$-partially traced by a point in $X$. Let $n\geq M$ be sufficiently large, as in Definition \ref{def:partial-specification-of-periods-P}. Using the surjectivity of $T$ find $y_0\in X$ such that $T^n(y_0)=y$.
    
     Consider the specification $\{T^{[a_i,b_i)}(x_i)\}_{i=1}^2$ where $a_1=0$, $b_1=1$, $a_2=n+1$, $b_2=n+2$, and $x_1=x$, $x_2=y_0$. Using the partial specification we obtain $z\in X$ such that $\rho(z,x)<\eps$ and $\rho(T^n(z),T^n(y_0))<\eps$. Thus, we have found $z\in T^{-n}(V)\cap U$ and $(X,T)$ is topologically mixing. 
\end{proof}

\begin{prop}\label{prop:infinite-partial-spec}
Let $(X,T)$ have the partial specification property. Then for every $\eps>0$ there exists $M\geq 1$ such that each infinite $M$-spaced specification $\{T^{[a_i,b_i)}(x_i)\}_{i=1}^\infty$ can be $\eps$-partially traced by a point $z \in X$. 
\end{prop}
\begin{proof}
Fix $\eps>0$. Find $M\geq 1$ such that every finite $M$-spaced specification can be $\eps/2$-partially traced by a point in $X$. Let $\{T^{[a_i,b_i)}(x_i)\}_{i=1}^\infty$ be an \emph{infinite} $M$-spaced specification. For $r\geq 1$, let $z_r\in X$ be a point that $\eps/2$-partially traces the  $M$-spaced specification $\{T^{[a_i,b_i)}(x_i)\}_{i=1}^r$. We know that $[a_1,b_1)$ has finite length and there are only finitely many possible choices of $\Lambda_1^{(1)}\subset [a_1,b_1)$ satisfying
\begin{equation*}
	\dfrac{\Lambda_1^{(1)}}{b_1-a_1}>1-\eps.
\end{equation*}
Therefore, we can find $\Lambda_1^{(1)}\subset [a_1,b_1)$ and a subsequence $\{z_r^{(1)}\}_{r=1}^\infty\subset \{z_r\}_{r=1}^\infty$ such that for each $r\geq 1$ we have 
\begin{equation*}
	\dfrac{\Lambda_1^{(1)}}{b_1-a_1}>1-\eps \text{ and } \Lambda_1^{(1)}=\{ a_1\leq j < b_1 \mid \rho(T^j(z_r^{(1)}),T^j(x_1))<\eps\}.
\end{equation*}
Now, by repeating the same procedure, for every $i\geq 2$ we find $\Lambda_i^{(i)}\subset [a_i,b_i)$ and a subsequence $\{z_r^{(i)}\}_{r=1}^\infty\subset \{z_r^{(i-1)}\}_{r=1}^\infty$ such that
\begin{equation*}
	\dfrac{\Lambda_i^{(i)}}{b_i-a_i}>1-\eps \text{ and } \Lambda_i^{(i)}=\{ a_i\leq j < b_i \mid \rho(T^j(z_r^{(i)}),T^j(x_i))<\eps\}.
\end{equation*}
Therefore, for $n\geq1$, we obtain a point $z_n^{(n)}$ that partially traces the specification $\{T^{[a_i,b_i)}(x_i)\}_{i=1}^{k}$ along $\Lambda_1^{(1)},\ldots,\Lambda_k^{(k)}$ for all $k\leq n $. Now, we pass to the subsequence of $\{z_n^{(n)}\}_{n=1}^\infty$ if necessary to find $z\in X$ such that $\lim_{n\to\infty}z_n^{(n)}=z$ and the point $z$ $\eps$-partially traces the \emph{infinite} specification $\{T^{[a_i,b_i)}(x_i)\}_{i=1}^\infty$.
\end{proof}

%
%
%

%
\begin{thm}\label{thm:partial-spec=>asp}
    If $(X,T)$ has the partial specification property, then $(X,T)$ has the average shadowing property.
\end{thm}
\begin{proof}
    Let $\eps>0$. We want to find $\delta>0$ such that every $\delta$-average pseudo-orbit can be $\eps$-shadowed in average.
    
    \noindent 1. Using Proposition~\ref{prop:infinite-partial-spec}, we find $M\geq 1$ such that every $M$-spaced specification $\{T^{[a_i,b_i)]}(x_i) \}_{i=1}^\infty$ can be $\eps/8$-partially traced by a point in $X$.
    
    \noindent 2. By \cite[Lemma 4.7]{Yaari} we have that partial specification implies partial shadowing. Hence, we find $\delta_1>0$ such that every $\delta_1$ partial pseudo-orbit $\{z_i\}_{i=0}^r$ can be $\eps/8$-partially traced by a point in $X$.
    
    \noindent 3. Let $\delta<\delta_1^2$ and let $\bar{x}=\{x_i\}_{i=0}^\infty$ be a $\delta$-average pseudo-orbit. Using Proposition~\ref{prop:partial-p.o-avg-p.o} we see that there is $N_1\geq 1$ such that for every $r\geq N_1$ and $k\geq 0$ the sequence $\{x_{i+k}\}_{i=0}^r$ is a $\delta_1$-partial pseudo-orbit.
    
    \noindent 4. Fix $r\geq N_1$ such that $\dfrac{M+r\frac{\eps}{2}}{M+r}<\dfrac{3}{4}\eps$. For each $j\geq 0$ define $a_j=j\cdot (M+r)$ and $b_j=a_j+r$. Then for every $j\geq 0$ we find $z^{(j)}\in X$ such that the $\delta_1$-partial pseudo-orbit $\{x_{i+a_j}\}_{i=0}^r$ is $\eps/8$-partially traced by $z^{(j)}$.
    
    \noindent 5. Observe that $\left\{T^{[a_j,b_j)}(z^{(j)})\right\}_{j=0}^\infty$ is an $M$-spaced specification. Therefore, there exists $z\in X$ that $\eps/8$-partially traces it.
    
    \noindent 6. Let $k\geq 0$. Consider the set of $M+r$ indices $\{a_k, a_k+1, \ldots , a_{k+1}-1\}$. Define
    \[ A=\{ a_k\leq n<b_k\:\colon\: \rho (T^n(z),T^n(z^{(k)}))<\frac{\eps}{8}\} .\]
    By point (5) of this proof $|A|\geq r(1-\frac{\eps}{8})$. Consequently, of the first $r$ terms of $A$ no more than $r\frac{\eps}{8}$ do not belong to $A$. Define
    \[ B=\{ a_k\leq n<b_k\:\colon\: \rho (T^n(z^{(k)}),x_n)<\frac{\eps}{8}\} .\]
    By point (4) of the proof $|B|\geq r(1-\frac{\eps}{8})$. Consequently, $|A\cap B|\geq r(1-\frac{\eps}{4})$ and for at least that many indices we have $\rho (T^n(z),x_n)<\frac{\eps}{4}$. Recall that diam$X=1$ and hence for the rest of the indices we have $\rho (T^n(z),x_n)\leq 1$. It follows that
    \[\frac{1}{M+r} \sum_{a_k\leq n<a_{k+1}}\rho (T^n(z),x_n)\leq \frac{1}{M+r}\left(r(1-\frac{\eps}{4})\frac{\eps}{4}+M+r-r(1-\frac{\eps}{4})\right)<\frac{M+r\frac{\eps}{2}}{M+r}<\frac{3}{4}\eps.\]
Given that the length of each such segment is constant, we finally obtain
    \[
        \rho_B\left(\bar{x},\underline{z}_T\right)=\limsup\limits_{n\to\infty}\dfrac{1}{n}\sum_{i=0}^{n-1} \rho(x_i,T^i(z))<\eps.\qedhere
    \]
\end{proof}


We now turn our attention to:

\begin{thm}\cite[Theorem 8.2]{BJK}
If $(X,T)$ is a surjective topological dynamical system with the vague specification property, then the ergodic measures are dense among all invariant measures.
\end{thm}

We aim to replace the assumption of the vague specification property in this theorem with the average shadowing property. To this end observe that the authors only need chain mixing, the average shadowing property, and \cite[Proposition 8.1]{BJK} in their proof. But surjectivity of $T$ plus the average shadowing property implies chain mixing by \cite[Lemma 3.1]{KKO}. As for Proposition 8.1, its proof, again, uses only chain mixing and the average shadowing property. Hence, while not stated explicitely, \cite{BJK} contains the proof of:

\begin{thm}\label{thm:asp=>measure-density}
If $(X,T)$ is a surjective topological dynamical system with the average shadowing property, then the ergodic measures are dense among all invariant measures.
\end{thm}

Combining this with Theorem \ref{thm:partial-spec=>asp} we obtain:

\begin{cor}\label{cor:partial-spec=>measure-density}
    Let $(X,T)$ be a surjective topological dynamical system with the partial specification property. Then the set of ergodic measures $\MT^e(X)$ is dense in the space of invariant measures $\MT(X)$.
\end{cor}

\begin{example}
    Besicovitch completeness does not imply the AASP. The simplest example is an identity on a two-point space.

    For a nontrivial example on an infinite space, one can take $A=\{0,1\}$ and the subshift $X=\{ 0^n1^\infty \colon n\geq0\} \cup\{0^\infty\}\subset A^\infty$. Observe that for $p,q\in X$ $$\rho_B^\sigma (p,q)=\left\{\begin{array}{cl}
       0  & p=q=0^\infty \mbox{ or } p\neq 0^\infty\neq q, \\
       1  & \mbox{otherwise}
    \end{array}\right.$$
    and therefore $(X,\sigma)$ is Besicovitch complete. It does not have the AASP, however, as the sequence
    $$0^\infty,1^\infty,0^\infty,0^\infty,1^\infty,1^\infty,0^\infty,0^\infty,0^\infty,0^\infty,1^\infty,1^\infty,1^\infty,1^\infty,\ldots,\underbrace{0^\infty,\ldots,0^\infty}_{2^n},\underbrace{1^\infty,\ldots,1^\infty}_{2^n},\ldots$$ is an asymptotic average pseudo-orbit of $\sigma$ that is not asymptotically shadowed in average by any of its true orbits.

\end{example}

%

\begin{example}\label{exm:not-Besicovitch-complete}
There exists a topological dynamical system $(X,T)$ with compact $X$ that is not Besicovitch complete.

Consider the cylinder $C=\{(z,t)\in\mathbb{C}\times\R\colon |z|=1\}$ with the geodesic metric derived from the euclidean metric in $\R^3$. Define $S_0=\{(z,0)\in C: z\in\mathbb{C}\}$ and for $n\geq 1$ let $S_n=\{(z,1/n)\in C: z\in\mathbb{C}\}$.

Define the sequence $\{r_n\}_{n=1}^\infty$ as $0, 1/2, 0, 1/4, 2/4, 3/4, 0, 1/8, \ldots$ (we write in increasing order all fractions of the form $k/2^n$, where $0\leq k<2^n$, starting with $n=1$ and increasing $n$ by one once we are done).

For $n\geq 1$ let $A_n\subset S_n$ be an arc of length 1 with the left end point $(e^{2\pi i r_n},1/n)$ and the right end point $(e^{2\pi i r_n+i},1/n)$.

The space $X$ will consist of the whole $S_0$ and finitely many points chosen from each $A_n$ for $n\geq 1$. We will construct it in 4 steps.

\noindent {\bf 1.} We include the entire $S_0$ in $X$ as a set of fixed points of $T$.

\noindent {\bf 2.} Let $a_1^1$ be the left endpoint of $A_1$. For $k> 1$ we define:

\begin{itemize}
\item $a^1_k$ is the left end point of $A_k$,
\item $a^2_k$ is on $A_k$ at the distance $1/2$ to the right of $a^1_k$,
\item $a^3_k$ is on $A_k$ at the distance $1/4$ to the right of $a^2_k$,

\ldots

\item $a^k_k$ is on $A_k$ at the distance  $1/2^{k-1}$ to the right of $a^{k-1}_k$.
\end{itemize}
Additionally, for $k>1$ let $a^k_1$ be the right end point of $A_k$ and let $a^k_2, \ldots, a^k_{k-1}$ be positioned, in this order, at even intervals between $a^k_1$ and $a^k_k$.

We add $\{a^k_n\}_{k,n\geq 1}$ to $X$ and set $T(a^k_n)=a^k_{n+1}$ (see Figure 1).

\begin{figure}
    \centering
    \includegraphics[width=0.9\linewidth]{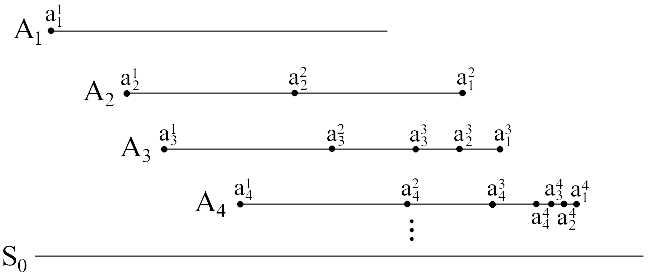}
    \caption*{Figure 2}
\end{figure}

\noindent {\bf 3.} Observe that, at the moment:

\noindent {\bf a)} $X$ is a compact space,

\noindent {\bf b)} the map $T$ is continuous,

\noindent {\bf c)} given any $k,m\geq 1$ the distance between $T^n(a^k_1)$ and $T^n(a^m_1)$ is, starting at some $n\geq 1$, constantly equal to $|(1/2)^{k-1}-(1/2)^{m-1}|$. Therefore, $\rho_B^T(a^k_1,a^m_1)=|(1/2)^{k-1}-(1/2)^{m-1}|$ and $\{a^n_1\}_{n\geq 1}$ is a Besicovitch-Cauchy sequence.

\noindent {\bf 4.} Finally, observe that no point in $X$ shadows $\{a^n_1\}_{n\geq 1}$ in the sense of Definition \ref{Besicovitchshadowing}:

\noindent {\bf a)} points of the form $a^n_1$ are excluded by virtue of $\rho_B^T(a^k_1,a^m_1)=|(1/2)^{k-1}-(1/2)^{m-1}|$,

\noindent {\bf b)} for points of the form $a^n_s$ with $s>1$ the distance between $T^r(a^n_s)$ and $T^r(a^n_1)$ tends to zero as $r$ tends to infinity; this coupled with $\rho_B^T(a^k_1,a^m_1)=|(1/2)^{k-1}-(1/2)^{m-1}|$ and the triangle inequality spoils any chance at shadowing,

\noindent {\bf c)} for $b\in S_0$ consider the orthogonal projections of $T^r(a^n_s)$ on $S_0$: by virtue of the choice of $r_n$ these points divide naturally into sets of finite blocks of length $2, 4, 8, \ldots$, with each block consisting of points that are positioned symmetrically around $S_0$. This forces $\rho_B^T(b,a^n_1)$ to be at least as much as the quarter of the length of $S_0$, as within each block the distance between $b$ and these projections will have this value as the average.
\end{example}

%

%
%
\begin{question}
    Does there exist a topological dynamical system $(X,T)$ that has the partial shadowing property but is not Besicovitch-complete?
\end{question}

%


    \bibliographystyle{plain}
    \bibliography{bibliography.bib}
\end{document}